\documentclass{jT}
\usepackage{amssymb,amsmath}
\usepackage[english]{babel}
\usepackage{url}
\newcommand{\afrac}[2]{\genfrac{}{}{0pt}{}{#1}{#2}}
\newcommand{\Q}{\mathbf{Q}}
\newcommand{\R}{\mathbf{R}}

\newcommand{\Z}{\mathbf{Z}}
\newcommand{\Fp}{\mathbf{F}_p}
\newcommand{\Fl}{\mathbf{F}_\ell}
\newcommand{\F}{\mathbf{F}}
\newcommand{\fp}{\mathfrak{p}}
\newcommand{\glmodl}{{\rm GL}_2(\Fl)}
\newcommand{\pglmodl}{{\rm PGL}_2(\Fl)}
\newcommand{\pslmodl}{{\rm PSL}_2(\Fl)}
\newcommand{\rhobarl}{\rho_{E,\ell}}
\newcommand{\frob}{\varphi_\fp}
\newcommand{\frobbarl}{\varphi_{\fp,\ell}}
\newcommand{\kron}[2]{\left(\frac{#1}{#2}\right)}
\newcommand{\inkron}[2]{\genfrac {(}{)}{0.9pt}{}{#1}{#2}}
\newcommand{\leftvert}{\vert\hspace{0.3pt}}
\newcommand{\cO}{\mathcal{O}}
\def\disc{\operatorname{disc}}
\def\Gal{\operatorname{Gal}}
\def\Aut{\operatorname{Aut}}
\def\tr{\operatorname{tr}}
\def\det{\operatorname{det}}

\newtheorem*{thm1}{Theorem 1}
\newtheorem*{thm2}{Theorem 2}
\newtheorem{lemma}{Lemma}
\newtheorem{proposition}{Proposition}

\begin{document}

\title[A local-global principle for rational isogenies of prime degree]{A local-global principle for\\rational isogenies of prime degree}
\author{\sc Andrew V. Sutherland}
\address{Andrew V. Sutherland\\Department of Mathematics\\Massachusetts Institute of Technology\\77 Massachusetts Avenue\\Cambridge, MA  02139}
\email{drew@math.mit.edu}
\urladdr{\url{http://math.mit.edu/~drew}}

\maketitle

\begin{resume}
Soit $K$ un corps de nombres. Nous \'etudions un principe local-global
pour les courbes elliptiques $E/K$ admettant ou non une isog\'enie
rationnelle de degr\'e premier $\ell$. Pour des corps $K$ convenables
(dont $K=\mathbb Q$), nous d\'emontrons ce principe pour tout $\ell\equiv
1\bmod 4$ et tout $\ell<7$ mais exhibons une courbe elliptique
d'invariant modulaire $2268945/128$ comme contre-exemple pour $\ell=7$.
Nous montrons alors qu'il s'agit du seul contre-exemple \`a isomorphisme
pr\`es lorsque $K=\mathbb Q$.
\end{resume}

\begin{abstr}
Let $K$ be a number field.  We consider a local-global principle for elliptic curves $E/K$ that admit
(or do not admit) a rational isogeny of prime degree $\ell$.
For suitable $K$ (including $K=\Q$), we prove that this principle holds for all $\ell\equiv 1\bmod 4$, and for $\ell < 7$, 
but find a counterexample when $\ell=7$ for an elliptic curve with $j$-invariant $2268945/128$.
For $K=\Q$ we show that, up to isomorphism, this is the only counterexample.
\end{abstr}

\section*{Introduction}

Let $E$ be a non-singular elliptic curve defined over a number field $K$, and let~$\ell$ be a prime number.
We say that $E$ admits an $\ell$-isogeny over $K$ if there is an isogeny $\phi\colon E\to E'$ of degree $\ell$
defined over $K$.  If $\fp$ is a prime of~$K$ where $E$ has good reduction, we say that $E$ admits
an $\ell$-isogeny locally at~$\fp$ if the reduction of $E$ modulo $\fp$ admits an isogeny of degree $\ell$
defined over the residue field.

If $E$ admits an $\ell$-isogeny over $K$, then $E$ necessarily admits an $\ell$-isogeny locally everywhere,
that is, at every prime of good reduction.  We ask the converse:

\begin{center}
\emph{If $E$ admits an $\ell$-isogeny locally everywhere,\\ must $E$ admit an $\ell$-isogeny over $K$?}
\end{center}

We can immediately answer this question with a counterexample.
Consider the elliptic curve $E/\Q$ defined by the Weierstrass equation
\begin{equation}\label{E2450A}
y^2 + xy = x^3 - x^2 - 107x - 379.
\end{equation}
This curve admits a $7$-isogeny locally at every prime of good reduction (and over $\R$) but it does not admit a $7$-isogeny over~$\Q$.
We note that $E$ does admit a 7-isogeny over a quadratic extension of~$\Q$.
Our first theorem implies that this is necessarily the case, as is the fact that we used a prime $\ell\equiv 3 \bmod 4$.

\begin{thm1}
Assume $\sqrt{\inkron{-1}{\ell}\ell}\notin K$, and suppose $E/K$ admits an $\ell$-isogeny locally at a set of primes with density one.
Then~$E$ admits an $\ell$-isogeny over a quadratic extension of $K$, and if $\ell\equiv 1 \bmod 4$ or $\ell < 7$, then $E$ admits an $\ell$-isogeny over~$K$.
\end{thm1}

Whether an elliptic curve  admits an $\ell$-isogeny over $K$ (or not) depends only on its
$j$-invariant $j(E)$, which uniquely identifies its isomorphism class over any algebraic
closure of~$K$.
Let us call a pair $(\ell,j(E))$ exceptional if $E/K$ admits an
$\ell$-isogeny locally everywhere, but not over~$K$.
The pair $(7,\ 2268945/128)$ corresponds to the counterexample above.
\mbox{Theorem}~1 admits the possibility
of infinitely many exceptional pairs, and this occurs, for example,
when $K=\Q(i)$.  However, it does not happen when $K=\Q$.

\begin{thm2}
The pair $(7,\ 2268945/128)$ is the only exceptional pair for $\Q$.
\end{thm2}

The analogous local-global question for $\ell$-torsion was addressed by Katz in~\cite{Katz:TorsionPoints}.
In general, an elliptic curve $E/K$ with non-trivial $\ell$-torsion locally everywhere may have
trivial $\ell$-torsion over $K$.  However, Katz shows that such an $E$ must be
rationally isogenous to an elliptic curve which has non-trivial $\ell$-torsion over $K$.  He proves this
by reducing the problem to a purely group-theoretic statement, an approach advocated by Mazur \cite[p.~483]{Katz:TorsionPoints}.

We take a similar line in our treatment of Theorem 1, using a Galois representation to reduce the problem
to a question regarding the structure of certain subgroups of $\glmodl$, which we are then able to address
with purely elementary methods.  The proof of Theorem 2 requires more, and here we also use the theory of complex multiplication and,
crucially, a theorem of Parent~\cite{Parent:Triviality} that characterizes the rational points on the modular
curve $X_0^+(\ell^2)(\Q)$ for certain values of $\ell$.

\section{Preliminaries}

\subsection{Galois Representations}
We follow the notation in \cite{Mazur:GaloisRepresentations}.  Let us fix a number field~$K$
with algebraic closure $\overline{K}$.  If $S$ is a finite set of non-archimedean primes of~$K$,
let $\overline{K}_S$ denote the maximal algebraic extension of $K$ in $\overline{K}$ unramified outside of~$S$.
Given an elliptic curve $E/K$ and a prime number $\ell$, the absolute Galois group $G_K=\Gal(\overline{K}/K)$
acts on the group of $\ell$-torsion points $E[\ell]$.  Provided $S$ contains the primes that divide~$\ell$, and
the primes where $E$ has bad reduction, this action factors through $G_{K,S}=\Gal(\overline{K}_S/K)$, yielding a representation
\[
\rhobarl\colon G_{K,S}\rightarrow\Aut(E[\ell])\cong{\rm GL}_2(\Z/\ell\Z)\cong \glmodl.
\]

If $\fp$ is a non-archimedean prime of $K$ not in $S$ with residue field $k_\fp$, let~$\frob$ denote the
$\leftvert k_\fp\vert$-power Frobenius automorphism, which we may view as an element of $G_{K,S}$ with the understanding that it
is determined only up to conjugacy.  We identify the image of $\frob$ under $\rhobarl$ as a 
conjugacy class $\frobbarl$ of $\glmodl$, and we have
\[
\det(\frobbarl)\equiv \leftvert k_\fp\vert \bmod\ell,\qquad \tr(\frobbarl)\equiv \leftvert k_\fp\vert+1-\leftvert E(k_\fp)\vert\bmod \ell.
\]
The Chebotarev density theorem implies that each conjugacy class $\frobbarl$ in~$G$ arises for a set of primes with density $|\frobbarl|/|G| > 0$.

\subsection{Subgroups of $\glmodl$}
We recall some of the classical subgroups of $\glmodl$.
A \emph{Cartan} subgroup is an absolutely semisimple maximal abelian subgroup, of which there are two types.
A \emph{split} Cartan subgroup is isomorphic to $\F_\ell^*\times\F_\ell^*$ and
conjugate to the group of diagonal matrices.  A \emph{non-split} Cartan subgroup is isomorphic to $\F_{\ell^2}^*$
and conjugate to a group $C_{ns}$ we may define as follows:
for $\ell=2$ let $C_{ns}$ be the unique subgroup of order 3, and for $\ell>2$ fix a quadratic non-residue $\delta\in\Fl^*$
and let
\[
C_{ns} = \left\{\left(\afrac{x}{y}\thinspace\afrac{\delta y}{x}\right):x,y\in\Fl, (x,y)\ne(0,0)\right\}.
\]
Every Cartan subgroup has index 2 in its normalizer and contains the group of scalar matrices.

The semisimple elements of $\glmodl$ are those of order prime to $\ell$.  The following proposition
characterizes the semisimple subgroups of $\glmodl$ in terms of their images in $\pglmodl$.

\begin{proposition}\label{prop:SSS}
Let $G$ be a subgroup of $\glmodl$ of order prime to $\ell$, and let $H$ be the image of $G$ in $\pglmodl$.
Then exactly one of the following holds:
\begin{enumerate}
\item[(a)]
$H$ is cyclic and $G$ lies in a Cartan subgroup;
\item[(b)]
$H$ is dihedral and $G$ lies in the normalizer of a Cartan subgroup but not in the Cartan subgroup itself;
\item[(c)]
$H$ is isomorphic to $A_4$, $S_4$, or $A_5$.
\end{enumerate}
Here $A_n$ and $S_n$ denote the alternating and symmetric groups on $n$ letters, respectively.
\end{proposition}
\begin{proof}
See \cite[Thm.~XI.2.3]{Lang:ModularForms} or \cite{Serre:GaloisRepresentations}.
\end{proof}

Let $\Omega$ denote the set of linear subspaces of $\Fl^2$.  The group $\glmodl$
acts on~$\Omega$, and the induced action of $\pglmodl$ is faithful, since only scalar matrices act trivially.
For an element or subgroup~$g$ of $\glmodl$ or $\pglmodl$,
we let $\Omega/g$ denote the set of $g$-orbits of $\Omega$, and define $\Omega^g$ to be
the set of elements fixed by~$g$.  We recall the orbit-counting lemma,
\[
|\Omega/G| = \frac{1}{|G|}\sum_{g\in G}|\Omega^g|,
\]
and understand that the size of each orbit in $\Omega/G$ must divide $|G|$.

We can use Proposition~\ref{prop:SSS} to characterize the action of each element of $\glmodl$ on $\Omega$.
This yields Proposition~\ref{prop:action}, which encapsulates the group-theoretic content of two results credited to Atkin \cite[\S 6]{Schoof:ECPointCounting2}.  For each $h$ in $\pglmodl$, let $\sigma(h)$
denote the sign of $h$ as a permutation of $\Omega$ (thus for $\ell>2$ we have $\sigma(h)=1$ if and only if $h\in\pslmodl$).

\begin{proposition}\label{prop:action}
Let $g\in\glmodl$ have image $h$ in $\pglmodl$ with order $r$, let $k=|\Omega^h|$, and let $s=|\Omega/h|$.
Then $k$ is $0, 1, 2$, or $\ell+1$, and the $s-k$ non-trivial $h$-orbits have size $r$.
When $\ell>2$ we also have $\sigma(h)=(-1)^s$.
\end{proposition}
\begin{proof}
The proposition clearly holds for $\ell=2$, so we assume $\ell > 2$.
When $r=1$, we have $k=s=\ell+1$ and $\sigma(h)=1=(-1)^s$.
When $r=\ell$, there are exactly two $h$-orbits, of sizes $1$ and~$\ell$, and we have $k=1$, $s=2$,
and $\sigma(h) = 1 = (-1)^s$.  The proposition holds in both cases.

Otherwise the cyclic group $H=\langle h \rangle$ has order $r$ prime to $\ell$ and we are in case (a) of Proposition~\ref{prop:SSS}.  Thus $g$ lies in a Cartan subgroup $C$, whose image in $\pglmodl$ has order $\ell-1$ or $\ell+1$, as $C$ is split or non-split (resp.).  We consider the two cases.

If $C$ is split then $g$ is diagonalizable, so $k=|\Omega^h|=|\Omega^g|=2$, and the same is true for every non-trivial element of $H$. The orbit-counting lemma yields
\[
s = \frac{1}{r}\sum_{h'\in H}\left|\Omega^{h'}\right| = \frac{1}{r}\bigl((r-1)2+\ell+1\bigr) = 2 + \frac{\ell-1}{r}.
\]
The sizes of the $(\ell-1)/r$ non-trivial orbits all divide $r$ and sum to $\ell-1$, hence they are all equal to~$r$.
If $r$ is odd then $s$ is even and $\sigma(h)=1=(-1)^s$, and if $r$ is even then $\sigma(h)=(-1)^{s-2}=(-1)^s$.
Thus the proposition holds when $C$ is split.

If $C$ is non-split then $g$ has no eigenvalues in $\Fl$, hence $k=|\Omega^h|=|\Omega^g|=0$,
and the same is true for every non-trivial element of $H$.
The orbit-counting lemma yields $s=(\ell+1)/r$, and every orbit
must have size~$r$. 
If $r$ is odd, then $s$ is even and $\sigma(h)=1=(-1)^s$, otherwise $r$ is even and $\sigma(h)=(-1)^s$.
Thus the proposition also holds when $C$ is non-split.
\end{proof}

\section{Proof of Theorem 1}

We first prove a group-theoretic lemma from which Theorem~1 will follow.
In terms of the elliptic curve $E$ of Theorem~1, the group $G$ in Lemma~\ref{lemma:main} is the image of the Galois representation $\rhobarl$.
The hypothesis of the lemma is met precisely when $E$ admits an $\ell$-isogeny locally everywhere but not globally; the lemma then imposes specific constraints on $G$ and $\ell$.

\begin{lemma}\label{lemma:main}
Let $G$ be a subgroup of $\glmodl$ whose image $H$ in $\pglmodl$ does not lie in $\ker \sigma$.
Suppose $|\Omega^g|>0$ for all $g\in G$, but $|\Omega^G|=0$.
The following hold:
\begin{enumerate}
\item $H$ is dihedral of order $2n$, where $n>1$ is an odd divisor of $(\ell-1)/2$;\label{lemma:main1}
\item $G$ is properly contained in the normalizer of a split Cartan subgroup;\label{lemma:main2}
\item $\ell\equiv 3\bmod 4$;\label{lemma:main3}
\item $\Omega/G$ contains an orbit of size $2$.\label{lemma:main4}
\end{enumerate}
\end{lemma}
\begin{proof}
No subgroup of ${\rm GL}_2(\F_2)$ satisfies the hypothesis of the lemma, so we assume $\ell > 2$.
Let $G$ and $H$ be as in the lemma, and note that $\Omega^H=\Omega^G$, $\Omega/H=\Omega/G$,
and if $h\in H$ is the image of $g\in G$ then $\Omega^h=\Omega^g$.

We first show that $\ell$ does not divide $m=|H|$.  The orbit-counting lemma yields
\[
|\Omega/H| = \frac{1}{m}\sum_{h\in H}|\Omega^h| \ge \frac{1}{m}(\ell+m) > 1,
\]
since $|\Omega^h| > 0$ for all $h\in H$ and $|\Omega^h|=\ell+1$ when $h$ is the identity.
If $\ell\mid m$ then $H$ contains an element $h$ of order $\ell$ and $\Omega/h$
consists of two orbits, of sizes 1 and $\ell$.  These must also be the
orbits of $\Omega/H$, since $|\Omega/H|>1$.  But this contradicts our assumption that $|\Omega^H|=0$, thus $\ell\nmid m$.

Since $\ell\nmid m$ we must have $|\Omega^h|\ne 1$ for all $h\in H$, as may be seen from the
proof of Proposition~\ref{prop:action}, thus $|\Omega^h|=2$ for every non-trivial $h\in H$.
We note that $H$ cannot be a cyclic group~$\langle h \rangle$, for this would imply $\Omega^h=\Omega^H$, contrary to our hypothesis.

We now show that $H$ is not isomorphic to $A_4$, $A_5$, or $S_4$.  The kernel of $\sigma\colon H\to\{\pm 1\}$ must
be an index 2 subgroup of $H$.
Neither $A_4$ nor $A_5$ contain such a subgroup.
Proposition~\ref{prop:action} implies that for $h\in H$, the value $\sigma(h)=(-1)^s$ depends
only on the order $r$ of $h$, since either $r=1$ and $s=\ell+1$, or $r >1$ and $s=2+(\ell-1)/r$ (since $k=|\Omega^h|=2$).
But there is no non-trivial homomorphism from $S_4$ to $\{\pm 1\}$ with this property: the sequence 1,9,8,6 that
counts the elements of order 1,2,3,4 (resp.) in~$S_4$ has no subsequence whose sum is 12.  It follows that $H\not\cong S_4$.

Proposition~\ref{prop:SSS} then implies that $H$ is a dihedral group of order $2n$
for some $n > 1$.  If $n$ is even then $H$ contains $n+1$ elements of order 2,
but $\sigma$ cannot be constant on a subset of $H$ with size $n+1>|H|/2$, so $n$ is odd.
For an element $h\in H$ of order $n$,
Proposition~\ref{prop:action} implies that $s=2+(\ell-1)/n$ is an integer, thus $n$
divides $\ell-1$ and in fact divides $(\ell-1)/2$.
This completes the proof of (\ref{lemma:main1}).

We are in case (b) of Proposition~\ref{prop:SSS}, so $G$ lies in the normalizer $N(C)$ of a Cartan subgroup $C$,
which must be split because $n$ does not divide $2(\ell+1)$.
The group $G$ must be properly contained in $N(C)$, since $n$ properly divides $\ell-1$.  This proves (\ref{lemma:main2}).

If $\ell\equiv 1\bmod 4$, then for each of the $n$
elements of $H$ with order 2 we have $s=2+(\ell-1)/2$ with even parity in Proposition~\ref{prop:action},
and $s$ also has even parity when $h$ is the identity.  But this implies that $\sigma$ is constant on a subset of
$H$ with size $n+1 >|H|/2$, which is again a contradiction.  Therefore $\ell\equiv 3\bmod 4$, proving (\ref{lemma:main3}).

Let $h\in H$ have order $n$.  Proposition~\ref{prop:action} implies that 
$\Omega/h$ consists of two trivial orbits and $(\ell-1)/n$ orbits of size $n$.
It follows that if $t$ is the size of an orbit of $\Omega/H$, then $t$ is congruent to
0, 1, or 2 modulo $n$, and we also know that $t$ divides $|H|=2n$ and $t\ne 1$ (since $|\Omega^H|=0$).  Therefore either $t=2$
or $t$ is a multiple of $n$.  But $n$ does not divide $\ell+1$, so the size of at least one orbit in $\Omega/H$
is not divisible by $n$.  This orbit must have size $t=2$, which proves (\ref{lemma:main4}).
\end{proof}

\begin{thm1}
Assume $\sqrt{\inkron{-1}{\ell}\ell}\notin K$, and suppose $E/K$ admits an $\ell$-isogeny locally at a set of primes with density one.
Then~$E$ admits an $\ell$-isogeny over a quadratic extension of $K$, and if $\ell\equiv 1 \bmod 4$ or $\ell < 7$, then $E$ admits an $\ell$-isogeny over~$K$.
\end{thm1}
\begin{proof}
Let the finite set $S$ consist of the primes where $E$ has bad reduction and the primes that divide $\ell$.
Let $G$ be the image of $\rho_{E,\ell}\colon G_{K,S}\to\glmodl$, and let~$g\in G$.  By the Chebotarev density theorem,
the conjugacy class of $g$ is equal to $\frobbarl$ for a set of primes $\fp$ with positive density, thus
we may choose $\fp$ so that $g=\frobbarl$ and $E$ admits an $\ell$-isogeny locally at $\fp$.  The Frobenius
endomorphism $\frob$ fixes a linear subspace of $E[\ell]$, hence $\frobbarl$ fixes an element of $\Omega$.
Thus $|\Omega^g|>0$ for all $g\in G$.

If $|\Omega^G|>0$, then $G_{K,S}$ fixes a linear subspace of $E[\ell]$ which is the kernel of an $\ell$-isogeny defined over $K$ (and any quadratic extension).  The theorem holds in this case, so we assume $|\Omega^G|=0$.  No subgroup $G\subset {\rm GL}_2(\F_2)$ has $|\Omega^G|=0$ and $|\Omega^g|>0$ for all $g\in G$, so $\ell\ne 2$.

The hypothesis on $K$ implies that some element of $G$ has a non-square determinant, otherwise $G_{K,S}$ fixes the quadratic Gauss sum $\sum_{n=0}^{\ell-1}\zeta_\ell^{n^2}$, which is equal to $\pm\sqrt{\inkron{-1}{\ell}\ell}$; see, e.g., \cite{Ireland:ClassicalModern}.
It follows that the image of $G$ in $\pglmodl$ does not lie in the kernel of $\sigma$, and we may apply Lemma~\ref{lemma:main}.
Thus $\ell\equiv 3\bmod 4$ and $\ell \ne 3$ (by part (\ref{lemma:main1}) of the lemma), and $\Omega/G$ contains an orbit of size 2.  Let $x\in\Omega$ be an element of this orbit.  The stabilizer of~$x$ in $G_{K,S}$ is a subgroup of index 2, corresponding to a quadratic extension of $K$ over which $E$ admits an isogeny of degree $\ell$.
\end{proof}

We now show that subgroups of the form permitted by Lemma~\ref{lemma:main} do in fact exist.

\begin{proposition}\label{prop:construct}
Let $\ell\equiv 3 \bmod 4$ be a prime number greater then $3$.
There exists a subgroup $G$ of $\glmodl$ with the following properties:
\begin{enumerate}
\item
The determinant map from $G$ to $\Fl^*$ is surjective;\label{prop:construct1}
\item
$|\Omega^g|\ge 2$ for all $g\in G$;\label{prop:construct2}
\item
$|\Omega^G|=0$.\label{prop:construct3}
\end{enumerate}
\end{proposition}
\begin{proof}
Let $\alpha$ be a generator for $\Fl^*$ and for integers $i$ and $j$ define
\[
A(i,j) =  \left( \afrac{\alpha^i}{0}\thinspace\afrac{0}{\alpha^j} \right);\qquad B(i,j) = \left( \afrac{0}{\alpha^j}\thinspace\afrac{\alpha^i}{0} \right). \]
Let $G$ be the set of matrices $A(i,j)$ and $B(i,j)$ for which $i$ and $j$ have the same parity.
It is easily checked that $G$ is a group.  It contains the scalar matrices $A(i,i)$, and the matrix $B(0,0)$ whose determinant $-1$ is not a quadratic residue in $\Fl^*$, since $\ell\equiv 3\bmod 4$. Therefore (\ref{prop:construct1}) holds.  Clearly $A(i,j)$ is diagonalizable, and so is $B(i,j)$, with distinct eigenvalues $\alpha^{(i+j)/2}$ and $-\alpha^{(i+j)/2}$.  This implies (\ref{prop:construct2}).
The matrices $A(0,2)$ and $B(0,2)$ have no common eigenspace, which proves (\ref{prop:construct3}).
\end{proof}

The group in Proposition \ref{prop:construct} has a dihedral image in $\pglmodl$ of order~$2n$, where $n=(\ell-1)/2$.
A similar construction works for any odd $n\ge 3$ dividing $(\ell-1)/2$.

\section{A counterexample}

The existence of the subgroups prescribed by Lemma~\ref{lemma:main} and constructed in Proposition \ref{prop:construct} does not imply that they actually arise as the image of~$\rho_{E,\ell}$ for some elliptic curve $E/K$.  Indeed, if~$E$ does not have complex multiplication, then $\rho_{E,\ell}$ is known to be surjective for all sufficiently large~$\ell$, as shown by Serre in \cite{Serre:GaloisRepresentations}.  Surjectivity certainly precludes an exceptional image of the type permitted by Lemma~\ref{lemma:main}, and, as we will show in the proof of Theorem 2, so does complex multiplication.

But there is an elliptic curve $E/\Q$ and a prime $\ell$ for which the image of $\rho_{E,\ell}$ satisfies the hypothesis of Lemma~\ref{lemma:main}: the curve defined by equation~(\ref{E2450A}) of the introduction, with $j$-invariant $j(E)=2268945/128$, and $\ell=7$.

Let $\Phi_N(X,Y)$ denote the classical modular polynomial \cite[\S 5]{Lang:EllipticFunctions}.
For a field $F$ of characteristic not dividing $N$, an elliptic curve $E/F$ admits a cyclic isogeny of degree~$N$ defined over~$F$ if and only if $\Phi_N(X,j(E))$ has a linear factor in $F[X]$, as proven in~\cite{Igusa:ModularPolynomial}.  We note that an isogeny of prime degree is necessarily cyclic.
\vspace{2pt}

The irreducible factorization of $\phi(X) = \Phi_7(X,2268945/128)$ over $\Q[X]$ is given below:
\small
\begin{align*}
\phi(X) =\  &(X^2 -  1306315496294666865/2^{49}\thinspace X + 1567296563714555573025/2^{50})\\
                  &(X^3 - 405372852936775146868447415805\thinspace X^2\\
                   &\quad\medspace\medspace\thinspace\thinspace+ 55385584722536349330202265781434325/4\thinspace X - 17996436663345^3/8)\\ 
                  &(X^3 - 4209728442885/2\thinspace X^2 + 3961627130765133274725/2\thinspace X - 16205314545^3/8).
\end{align*}
\normalsize
The absence of a linear factor shows that $E$ does not admit a 7-isogeny over~$\Q$, but it does admit a 7-isogeny (two in fact) over a quadratic extension of $\Q$, as required by Theorem~1.  The discriminants of the three factors of $\phi(X)$ are each of the form $-7a^2/4^b$ for some positive integers $a$ and $b$.  It follows that the reduction of $\phi(X) \bmod p$ has two linear factors in $\Fp[X]$ for every odd prime $p$: either $-7$ is a quadratic residue mod $p$ and the quadratic factor splits in $\Fp[X]$, or it is not and both cubic factors split into a linear and a quadratic factor in $\Fp[X]$.  Thus $E$ admits a 7-isogeny locally at every prime of good reduction (all $p\nmid 70$).

It is easily verified that the cubic factors of $\phi(X)$ have the same splitting field, in which the quadratic factor also splits.  Its Galois group is isomorphic to $S_3$, which we may view as a subgroup of ${\rm PGL}_2(\F_7)$ via its action on the roots of $\phi(X)$.  Thus we have a dihedral subgroup of ${\rm PGL}_2(\F_7)$ with order $2n$, where $n=3$ divides $(\ell-1)/2$, as required by Lemma~\ref{lemma:main}.  Up to conjugacy, the image of $\rho_{E,7}$ in ${\rm GL}_2(\F_7)$ is precisely the group $G$ constructed in Proposition \ref{prop:construct}.

To determine whether there are any other exceptional pairs of the form $(7,j(E))$, we consider the moduli space of elliptic curves with the required level 7 structure.  As explained by Elkies in \cite[\S 4.2]{Elkies:EightfoldWay}, the corresponding modular curve may be constructed as a quotient of $X(7)$ by a suitable subgroup of ${\rm PSL}_2(\F_7)$.  We are interested in
elliptic curves $E$ for which the image $H$ of $\rho_{E,7}$ in ${\rm PGL}_2(\F_7)$ is a dihedral group of order 6 not contained in the kernel of $\sigma$.  Up to conjugacy there is precisely one such $H$, and its intersection with ${\rm PSL}_2(\F_7)$ is a cyclic group of order 3.

As shown in \cite{Elkies:EightfoldWay}, the quotient of $X(7)$ by such a group is isomorphic (over~$\overline{\Q}$) to $X_0(49)$, an elliptic curve.
We are interested in the twist of $X_0(49)$ by $\Gal(\Q(\sqrt{-7})/\Q)$, relative to the Fricke involution~$w_{49}$, since if $(7,j(E))$ is an exceptional pair for $\Q$, then $\Q(\sqrt{-7})$ is the unique quadratic extension of $\Q$ over which $E$ admits a 7-isogeny (by Theorem~1 there is such an extension, and from the proof of Theorem~1 it follows that if $\sqrt{-7}\not\in K$ then $(7,j(E))$ will remain exceptional for $K$).
If we make the cusp of $X_0(49)$ at infinity the origin, the cusp at zero is a rational point of order~2, and there are two non-cuspidal irrational 2-torsion points defined over $\Q(\sqrt{-7})$.
Our twist makes the two rational cusps irrational, and the irrational 2-torsion points are made rational and swapped by the $w_{49}$ involution.  An explicit computation by Elkies \cite{Elkies:X49twist} finds that this twist has the equation
\begin{equation}\label{X49twist}
-7y^2 = x^4 + 2x^3 - 9x^2 - 10x - 3,
\end{equation}
which is a genus 1 curve with rational points $(-1/2,\pm 1/4)$.  If $(x,y)$ is a rational solution to~(\ref{X49twist}),
then the rational map
\small
\[
f(x) = -(x-3)^3(x-2)(x^2+x-5)^3(x^2+x+2)^3(x^4-3x^3+2x^2+3x+1)^3 / (x^3-2x^2-x+1)^7
\]
\normalsize
yields the $j$-invariant of an elliptic curve that admits a 7-isogeny locally everywhere but not over~$\Q$.
Applying $f$ to either of the points $(-1/2,\pm 1/4)$ yields the $j$-invariant $2268945/128$ that we have already exhibited.

Taking either of $(-1/2, \pm1/4)$ as the origin, the curve in (\ref{X49twist}) is isomorphic over~$\Q$ to the elliptic curve with Weierstrass equation
\begin{equation}\label{49a3}
y^2 + xy = x^3 - x^2 - 107x + 552.
\end{equation}
This is curve 49a3 in Cremona's tables \cite{Cremona:Database}.  It has just two $\Q$-rational points, and these correspond
to the two solutions of (\ref{X49twist}) that we already know.  Thus $(7,\ 2268945/128)$ is the only exceptional pair for $\Q$ with $\ell=7$.

However, over a finite extension $K$ of $\Q$ the curve 49a3 may have infinitely many rational points, corresponding to infinitely many exceptional pairs $(7,j(E))$ for $K$.
Over $K=\Q(i)$, for example, the projective point $(-14:7+29i:1)$ on the curve 49a3 has infinite order, yielding infinitely many solutions to $(\ref{X49twist})$.
The first of these has $x$-coordinate $(7i-29)/58$, and in general, if $(u,v)$ is a solution to $(\ref{49a3})$ then there is a solution to $(\ref{X49twist})$ with $x$-coordinate $(3u-v+42)/(u+2v)$.

\section{Proof of Theorem 2}

Recall that a pair $(\ell,j(E))$ is said to be exceptional for $K$ when $E/K$ admits an $\ell$-isogeny locally everywhere but not over $K$.

\begin{thm2}
The pair $(7,\ 2268945/128)$ is the only exceptional pair for $\Q$.
\end{thm2}
\begin{proof}
By Theorem 1 we need only consider exceptional pairs with $\ell \ge 7$ and $\ell\equiv 3\bmod 4$.
The case $\ell=7$ is treated above, so we assume $\ell > 7$.

We first show that if $(\ell,j(E))$ is an exceptional pair then $E$ cannot have complex multiplication (CM). Suppose the contrary. Then $E$ has CM by an imaginary quadratic order $\cO$, since we are in characteristic zero, and it must have class number $h(\cO)= 1$, since $E$ is defined over~$\Q$.  By Theorem~1 there is an $\ell$-isogenous elliptic curve $E'$ that is defined over a quadratic extension of $\Q$ but not over $\Q$ (since we are in an exceptional case).  The curve~$E'$ must have CM by an imaginary quadratic order $\cO'$ with class number $h(\cO')= 2$.  Since $E$ and $E'$ are $\ell$-isogenous and $h(\cO') > h(\cO)$, the order $\cO'$ must be properly contained in $\cO$ with index $\ell$.
Via \cite[Thm.~7.24]{Cox:ComplexMultiplication}, we may compute the ratio $h(\cO')/h(\cO)$ as
\[
\frac{h(\cO')}{h(\cO)} = \frac{1}{[\cO^*:\cO'^*]}\left(\ell - \kron{\disc(\cO)}{\ell}\right) \ge \frac{1}{3}\bigl(\ell-1\bigr).
\]
But we already know that $h(\cO')/h(\cO) = 2/1 = 2$, which yields a contradiction for $\ell > 7$.  Therefore $E$ does not have CM.

By Lemma 1, it suffices to consider pairs $(j(E),\ell)$ for which the image of $\rho_{E,\ell}$ in ${\rm GL}_2(\mathbb{F}_\ell)$ is contained in the normalizer of a split Cartan group, as explained in the proof of Theorem~1.  For such a pair, the $j$-invariant $j(E)$ corresponds to a rational non-cuspidal point on the modular curve $X_{\rm split}(\ell)(\Q)$, which is isomorphic over $\Q$ to $X_0^+(\ell^2)$, the quotient of $X_0(\ell^2)$ by the Fricke involution $w_{\ell^2}$.  By Theorem 1.1 of \cite{Parent:Triviality}, for all $\ell > 7$ congruent to $3\bmod 4$, the only rational non-cuspidal points on $X_0^+(\ell^2)(\Q)$ correspond to elliptic curves with CM.\footnote{This is also implied by the recent (and stronger) result in \cite{BiluParentRebolledo:Xsplit}, which addresses all $\ell\ne 13$.}  But we have shown that no curve with CM can arise in an exceptional pair, and the theorem follows.
\end{proof}

The argument used to rule out CM in the proof above works over any number field, and when $K\ne \Q$ we have $[\cO^*:\cO'^*]=1$, which covers the case $\ell=7$ as well.
\smallskip

\section{Acknowledgments}
I thank Nicholas Katz, Barry Mazur, and Ken Ribet for several helpful discussions, and Pierre Parent for directing my attention to his result in~\cite{Parent:Triviality}.  I am especially grateful to Noam Elkies for his explicit computations in the case $\ell=7$, and to John Cremona for identifying the constraint on~$K$ required by Theorem~1.

\providecommand{\bysame}{\leavevmode\hbox to3em{\hrulefill}\thinspace}


\begin{thebibliography}{10}

\bibitem{BiluParentRebolledo:Xsplit}
Yuri Bilu, Pierre Parent, and Marusia Rebolledo, \emph{Rational points on $X_0^+(p^r)$}, arXiv:1104.4641v1 (2011).

\bibitem{Cox:ComplexMultiplication}
David A. Cox, \emph{Primes of the form $x^2+ny^2$: {F}ermat, class field theory, and complex multiplication}, John Wiley and Sons, 1989.

\bibitem{Cremona:Database}
John Cremona, \emph{The elliptic curve database for conductors to $130000$},
  Algorithmic Number Theory Symposium--{ANTS VII} (F.~Hess, S.~Pauli, and
  M.~Pohst, eds.), Lecture Notes in Computer Science, vol. 4076,
  Springer-Verlag, 2006, pp.~11--29.

\bibitem{Elkies:EightfoldWay}
Noam~D. Elkies, \emph{The {K}lein quartic in number theory}, The Eightfold Way:
  The Beauty of {K}lein's Quartic Curve (Silvio Levy, ed.), Cambridge
  University Press, 2001, pp.~51--102.

\bibitem{Elkies:X49twist}
\bysame, private communication, March 2010.

\bibitem{Igusa:ModularPolynomial}
Jun-Ichi Igusa, \emph{Kroneckerian model of fields of elliptic modular
  functions}, American Journal of Mathematics \textbf{81} (1959), 561--577.

\bibitem{Ireland:ClassicalModern}
Kenneth Ireland and Michael Rosen, \emph{A classical introduction to modern
  number theory}, second ed., Springer-Verlag, 1990.

\bibitem{Katz:TorsionPoints}
Nicholas~M. Katz, \emph{Galois properties of torsion points on abelian
  varieties}, Inventiones Mathematicae \textbf{62} (1981), no.~3, 481--502.

\bibitem{Lang:ModularForms}
Serge Lang, \emph{Introduction to modular forms}, Springer, 1976.

\bibitem{Lang:EllipticFunctions}
\bysame, \emph{Elliptic functions}, second ed., Springer-Verlag, 1987.

\bibitem{Mazur:GaloisRepresentations}
Barry Mazur, \emph{An introduction to the deformation theory of {G}alois
  representations}, Modular forms and {F}ermat's last theorem, Springer, 1997.

\bibitem{Parent:Triviality}
Pierre J.~R. Parent, \emph{Towards the triviality of {$X_0^+(p^r)(\mathbb{Q})$}
  for $r > 1$}, Compositio Mathematica \textbf{141} (2005), 561--572.

\bibitem{Schoof:ECPointCounting2}
Ren\'{e} Schoof, \emph{Counting points on elliptic curves over finite fields},
  Journal de Th\'{e}orie des Nombres de Bordeaux \textbf{7} (1995), 219--254.

\bibitem{Serre:GaloisRepresentations}
Jean-Pierre Serre, \emph{Propri\'{e}t\'{e}s galoisiennes des points d'ordre
  fini des courbes elliptiques}, Inventiones Mathematicae \textbf{15} (1972),
  no.~2, 259--331.

\end{thebibliography}
\end{document}